\documentclass[11pt]{article}
\usepackage{amsmath,amssymb,amsthm}
\usepackage{enumerate}
\usepackage[mathscr]{eucal}
\usepackage[all,cmtip]{xy}
\usepackage{csquotes}
\usepackage{times}
\usepackage{hyperref}
\usepackage{enumerate}
\usepackage{float}

\usepackage{subcaption}
\makeatletter

\RequirePackage[dvips]{graphicx} \textheight 16cm
\usepackage{graphicx}
\usepackage{color}
\newcommand{\singlespacing}{\let\CS=\@currsize\renewcommand{\baselinestretch}{1}\tiny\CS}
\newcommand{\oneandahalfspacing}{\let\CS=\@currsize\renewcommand{\baselinestretch}{1.25}\tiny\CS}
\newcommand{\doublespacing}{\let\CS=\@currsize\renewcommand{\baselinestretch}{1.35}\tiny\CS}

\newtheorem{Theorem}{Theorem}[section]
\newtheorem{Definition}[Theorem]{Definition}

\newtheorem{Lemma}[Theorem]{Lemma}
\newtheorem{Proposition}[Theorem]{Proposition}
\newtheorem{Corollary}[Theorem]{Corollary}
\newtheorem{Remark}[Theorem]{Remark}
\newtheorem{Note}[Theorem]{Note}
\newtheorem{Example}[Theorem]{Example}

\begin{document}

\newcommand{\la}{\lambda}
\newcommand{\si}{\sigma}
\newcommand{\ol}{1-\lambda}
\newcommand{\be}{\begin{equation}}
\newcommand{\ee}{\end{equation}}
\newcommand{\bea}{\begin{eqnarray}}
\newcommand{\eea}{\end{eqnarray}}
\newcommand{\nn}{\nonumber}
\newcommand{\mc}{\multicolumn}
\newcommand{\bee}{\begin{eqnarray*}}
\newcommand{\eee}{\end{eqnarray*}}
\newcommand{\lb}{\label}
\newcommand{\D}{${\bf{D~}}$}
\newcommand{\C}{${\bf{C~}}$}
\newcommand{\nii}{\noindent}
\newcommand{\ii}{\indent}
\newcommand{\0}{${\bf 0}$}
\newcommand{\bpi}{\mbox{\boldmath$\pi$}}
\newcommand{\w}{\omega}
\newcommand{\bw}{\mbox{\boldmath$\omega$}}
\newcommand{\bet}{\mbox{\boldmath$
\beta$}}
\newcommand{\al}{\mbox{\boldmath$\alpha$}}
\newcommand{\n}{\mbox{\boldmath$\nu$}}
\singlespacing
\doublespacing
\title{ {\bf On local function, an algebraic approach}}
\author{ Monoj Kumar Das, Shyamapada Modak$^*$\\
 Department of  Mathematics, University of Gour Banga\\ Malda 732103, West Bengal, India \\
e-mail: spmodak2000@yahoo.co.in (S. Modak)\\ dmonojkr1@gmail.com (M. K. Das)\\
Orcid Id:  0009-0007-2107-9984 (M. K. Das)\\ 0000-0002-0226-2392 (S. Modak)}
\date{}

 \maketitle{}

 {\bf Abstract}:
The paper discuss the limit point concept of a subset in a group via ideal of the power set ring. This idea along with anti-ideal give the topological structure in a group. Homomorphic images of both ideal and anti-ideal are played the remarkable role to change the topological structure from one system to another system.

{\bf Mathematics Subject Classification (2020)}: 20C05, 20B35, 20K30, 54A05

{\bf Key words and phrases}:  Ideal, Anti-ideal, Group, Power set ring, Homomorphism

\section{\bf Introduction and Preliminaries }
Group and ring theories are known to us from 19th century. The ideal on a set $X$ is primarily defined by Kurtatowski \cite{KK} and Vaidyanathaswamy \cite{Vaidyanathaswamy}. They,  and a good number of mathematicians \cite{Hayasi1964,JANKOVIC,Jankovic2015,JH2023,KK,MIAH2025,Modak2012, Newcomb, SELIM} are discussed the ideals on the various field of mathematics \cite{KULCHHUM2023,JH2023, MS2023, SKJ2024, MN2007,KHATUN25,KHATUN2025}. However, for these discussion they used topology on the same set. The authors Das et. al \cite{DAS} have shown that this ideal is same the ideal on the power set ring. For a set $G$, $(2^G, \Delta, \cap)$ or simple  $2^G$, the power set ring (means, the set $2^G$ with two binary operation $\Delta$ and $\cap$, where `$\Delta$' is defined as $A\Delta B = (A\setminus B) \cup (B\setminus A)$).
We through this paper, consider the local function upon a group via ideal of the power set ring $2^G$, where $G$ is a group. Homomorphism and kernel will take a major role of this paper. Further we represent a new idea on a ring, which is called anti-ideal. Our another job through this paper is, anti-ideal's reflection on a group and its consequences.

Following lemma is an useful tool for the paper:

\begin{Lemma} \cite{DAS}
\label{-01}
Let $G$ be a group and $\mathcal{I}$ be an ideal on the power set ring $(2^G, \Delta, \cap)$. Then
 $\mathcal{I}$  is synonymous with the Kuratowski ideal on $G$.
\end{Lemma}

\section{Ideals on a group}
Local function or action of an ideal on a group is considered as follows:

\begin{Definition}
Let $G$ be a group and $A\subseteq G$ be any set. Define $A^{\propto}(\mathcal{I})(\text{or\;simply}\;A^{\propto})=\{g\in G: H\cap A\notin \mathcal{I},\;\forall\; H\in \mathcal{G}(g)\}$, where $\mathcal{G}(g)=\{H\subseteq G: \text{$H$ is a subgroup of $G$ and $g\in H$}\}$ and $\mathcal{I}$ is an ideal on $G$.
\end{Definition}

  Suppose $x\in A^{\propto}$, then for every subgroup $H$ containing $x$, $H\cap A\notin \mathcal{I}$. That is, $H\cap A\neq\varnothing$. One can consider the concept; $x$ is said to be a $\propto$-limit point of $A$ if $A$ intersects all subgroups containing $x$. This is called local function associated with group.

  Following are the properties and examples of this local function and its associated subjects.

We are starting with $A^{\propto}$ may not be a group (even if $A$ is a subgroup of $G$).

 \begin{Example}

  Consider the Klein $4$-group $G=K_{4}=\{e, a, b, c\}$ and $\mathcal{I}= \{\varnothing, \{c \}, \{e \}, \{e, c \} \}$. Then for the subgroup $A = \{e, a \}$, $a\in A^{\propto}$ but $e \notin A^{\propto}$.

 \end{Example}

 Even, if $\mathcal{I} = \mathcal{I}_{fin}$ on $2^G$, then $A^{^\propto}$ is not a subgroup in general.

 \begin{Example}
   Let $G=\mathbb{Z}$ and $\mathcal{I}=\mathcal{I}_{fin}$ (collection of all finite subsets of $G$). Then for the subset $A=\mathbb{Z}^{e}$, collection of even integers. Then $A^{\propto}=\mathbb{Z}\setminus \{0\}$ which is not a subgroup.
 \end{Example}

 \begin{Example}
   Let $G$ be any infinite group and $\mathcal{I}=\mathcal{I}_{fin}$. Take $A$ to be any finite subgroup (or subset) of $G$ of size greater than $2$ ( i.e $\{e,a\}$ with $a$ of order $2$). Then $A^{\propto}=\varnothing$.
 \end{Example}

\begin{Theorem}
  Let $G$ be a group and $g\in G$. Then $g\in A^{\propto}\iff <g>\cap A\notin\mathcal{I}$.
\end{Theorem}

\begin{proof}
  Let $g\in A^{\propto}$. In particular $<g>\in \mathcal{G}(g)\Rightarrow <g>\cap A\notin \mathcal{I}$.\\
  Conversely, let $<g>\cap A\notin \mathcal{I}$, for $g\in G$. Then for any $H\in \mathcal{G}(g)$, $<g>\cap A\subseteq H\cap A\Longrightarrow H\cap A\notin\mathcal{I}\Longrightarrow g\in A^{\propto}$.
\end{proof}

Note that this property has been managed through the hereditary property of ideal.

\begin{Proposition}
 Let $A$ be a subset of a group $G$. Then, for $g\in A^{\propto}\Longrightarrow g^{n}\in A^{\propto}$, where $n\in\mathbb{Z}$.
\end{Proposition}

\begin{Theorem}
\label{01}
 For two subsets $A$ and $B$ of  a group $G$ with $A\subseteq B$,  $A^{\propto}\subseteq B^{\propto}$.
\end{Theorem}

\begin{proof}
  Let $g\in A^{\propto}$, then for all $H\in\mathcal{G}(g)$, $A\cap H\notin\mathcal{I}$.\\
  Now, $A\subseteq B \Longrightarrow A\cap H\subseteq B\cap H\Longrightarrow B\cap H\notin\mathcal{I}\Longrightarrow g\in B^{\propto}$.
\end{proof}

\begin{Theorem}
  Let $G$ be a group and $A, B\subseteq G$. Then, $(A\cup B)^{\propto}=A^{\propto}\cup B^{\propto}$.
\end{Theorem}

\begin{proof}
Suppose $g\in (A\cup B)^{\propto}$ but $g\notin A^{\propto}\cup B^{\propto}$. Then, $g\notin A^{\propto}$ and $g\notin B^{\propto}\Longrightarrow <g>\cap A\in\mathcal{I}$ and $<g>\cap B\in\mathcal{I}\Longrightarrow (<g>\cap A)\cup (<g>\cap B)=<g>\cap (A\cup B)\in\mathcal{I}\Longrightarrow g\notin (A\cup B)^{\propto}$, contradiction. Thus, $g\in (A\cup B)^{\propto}\Longrightarrow g\in A^{\propto}\cup B^{\propto}$.

Reverse inclusion is obvious from the Theorem \ref{01}.

\end{proof}

\begin{Theorem}
  If $\mathcal{I}\subseteq \mathcal{J}$, where both  $\mathcal{I}$  and $\mathcal{J}$ are ideals on $2^{G}$,  then $A^{\propto}(\mathcal{J})\subseteq A^{\propto}(\mathcal{I})$.
\end{Theorem}

Proof is obvious from the fact every member of  $\mathcal{I}$ are also the member of $\mathcal{J}$.

\begin{Remark}
  Let $G$ be a group and $A, B\subseteq G$, then $(A\cap B)^{\propto} \neq A^{\propto}\cap B^{\propto}$.
\end{Remark}

\begin{Example}
  Let $G=\mathbb{Z}$, $\mathcal{I}=\mathcal{I}_{fin}$. Take $A=\mathbb{Z}^{e}$ and $B=\mathbb{Z}^{o}$, collection of  odd integers. Then $A^{\propto}=\mathbb{Z}\setminus \{0\}$ and $B^{\propto}=2\mathbb{Z}+1$ and hence $A^{\propto}\cap B^{\propto}=2\mathbb{Z}+1$. But $(A\cap B)^{\propto}=(\{\varnothing\})^{\propto}=\varnothing$.
\end{Example}

\begin{Theorem}
  Let $G$ be a group and $\mathcal{I}=\{H:\; H\subseteq G \;\text{and}\; e\notin H\}$. Then, for any subset $A\subseteq G$, $A^{\propto}$ is a subgroup of $G$ iff $A^{\propto} = G$.

\end{Theorem}

\begin{proof}
  Suppose $A^{\propto}$ is a subgroup of $G$. For $g\in G$ and any subgroup $K$ with $g\in K$, we have $e\in K\cap A\Longrightarrow K\cap A\notin \mathcal{I}\Longrightarrow g\in A^{\propto}$. Conversely, if $A^{\propto}=G$, then obviously it is a subgroup.
\end{proof}

\begin{Corollary}
  \begin{enumerate}[(i)]
    \item If $\{e\}\in \mathcal{I}$ (respectively $H\in \mathcal{I}$ with $e\in H$), then $e\notin A^{\propto}$ for every $A$, So $A^{\propto}$ is never a subgroup.
    \item If $\mathcal{I}=\{\varnothing \}$, then $\{e\}\notin\mathcal{I}$. If $e\in A$, then $A^{\propto}=G$, and if $e\notin A$, then $e\notin A^{\propto}$ and $A^{\propto}$ cannot be a subgroup of $G$.
  \end{enumerate}

\end{Corollary}
\begin{Theorem}
  Let $f:G\rightarrow G$ be group homomorphism and $\mathcal{I}$ is an ideal on $2^{G}$. Then,
  \begin{enumerate}[(i)]
    \item $f(\mathcal{I}) = \{f(I):\; I\in \mathcal{I} \}$ is an ideal on $G$.
    \item for any subset $A$ of $G$, $f(A^{\propto}(\mathcal{I}))\subseteq(f(A))^{\propto}(f(\mathcal{I}))$.
  \end{enumerate}

\end{Theorem}
\begin{proof}
$(i)$ Obvious from \cite{SELIM,JH2023,JS2023}.

$(ii)$ Let $y\in f(A^{\propto}(\mathcal{I}))$, then there exists $x\in A^{\propto}(\mathcal{I})$ such that $y=f(x)$. Then $<x>\cap A\notin \mathcal{I}$.\\ Now, $f(<x>\cap A)\subseteq f(<x>)\cap f(A)=<y>\cap f(A)\Longrightarrow <x>\cap A\subseteq f^{-1}(<y>\cap f(A))\Longrightarrow f^{-1}(<y>\cap f(A))\notin\mathcal{I}\Longrightarrow <y>\cap f(A)\notin f(\mathcal{I})\Rightarrow y\in (f(A))^{\propto}(f(\mathcal{I}))$. Therefore, $f(A^{\propto}(\mathcal{I}))\subseteq(f(A))^{\propto}(f(\mathcal{I}))$.
\end{proof}

Converse of the part $(ii)$ in the above theorem is not true in general:

\begin{Example}
  Take $G=\mathbb{Z}$, $\mathcal{I}=\mathcal{I}_{fin}$. Define $f:G\rightarrow G$ by $f(n)=2n$. Take $A=2\mathbb{Z}$, $A^{\propto}(\mathcal{I}_{fin})=\mathbb{Z}\setminus\{0\}\Rightarrow f(A^{\propto}(\mathcal{I}_{fin}))=2\mathbb{Z}\setminus \{0\}$. Again, $(f(A))^{\propto}(f(\mathcal{I}_{fin}))=(f(2\mathbb{Z}))^{\propto}(f(\mathcal{I}_{fin}))=(4\mathbb{Z})^{\propto}
  (f(\mathcal{I}_{fin}))=\mathbb{Z}\setminus \{0\}$.
\end{Example}

\begin{Theorem}
  Let $f:G\rightarrow G$ be an injective group homomorphism on $G$ and $\mathcal{I}$ be an ideal on $2^{G}$. Then for any subset $A$ of $G$, $f(A^{\propto}(\mathcal{I}))=(f(A))^{\propto}(f(\mathcal{I}))$.
\end{Theorem}

\begin{proof}
  Let $y\in f(A^{\propto}(\mathcal{I}))$, then there exists $g\in A^{\propto}(\mathcal{I})$ such that $y=f(g)$. Take a  subgroup $K \subseteq f(G)$ with $y\in K$ and let $H=f^{-1}(K)$ a subgroup of $G$ containing $g$. Since $g\in A^{\propto}(\mathcal{I})$, $H\cap A\notin \mathcal{I}\Longrightarrow f(H\cap A)\notin f(\mathcal{I})\Longrightarrow f(H)\cap f(A)\notin f(\mathcal{I})\Longrightarrow K\cap f(A)\notin f(\mathcal{I})\Longrightarrow y\in (f(A))^{\propto}(f(\mathcal{I}))$. Thus, $f(A^{\propto}(\mathcal{I}))\subseteq(f(A))^{\propto}(f(\mathcal{I}))$.

  Let $y\in (f(A))^{\propto}(f(\mathcal{I}))$, then for $g\in G$, $f(g)=y$. For a subgroup $H$ containing $g$, $f(H)$ is a subgroup containing $y$, So $f(H)\cap f(A)\notin f(\mathcal{I})\Longrightarrow f(H\cap A)\notin f(\mathcal{I})$ (as $f$ is injective) $\Longrightarrow H\cap A\notin \mathcal{I}\Longrightarrow g\in A^{\propto}(\mathcal{I})\Longrightarrow f(g)\in f(A^{\propto}(\mathcal{I}))\Longrightarrow y\in f(A^{\propto})\Longrightarrow (f(A))^{\propto}(f(\mathcal{I}) \subseteq f(A^{\propto}(\mathcal{I}))$. Therefore $f(A^{\propto}(\mathcal{I}))=(f(A))^{\propto}(f(\mathcal{I}))$.
\end{proof}

\begin{Corollary}
  Suppose $G$ be a group and $f: G\rightarrow G$ be a group homomorphism. Then,
    $(ker f)^{\propto}=G$, if   $\{e\}\notin\mathcal{I}$.

\end{Corollary}
\begin{Theorem}
 Let $G$ be a group and $\mathcal{I}$ be an ideal on $2^{G}$. Then, for any $A\subseteq G$ and $I\in\mathcal{I}$, $(A\cup I)^{\propto}=A^{\propto}=(A\setminus I)^{\propto}$.
\end{Theorem}

\begin{proof}
Let $I\in\mathcal{I}$ and $A\subseteq G$ be any subset of $G$. Since $A\subseteq (A\cup I)$, $A^{\propto}\subseteq (A\cup I)^{\propto}$.\\
  Let $g\in (A\cup I)^{\propto}$. Then $<g>\cap (A\cup I)\notin \mathcal{I}\Longrightarrow (<g>\cap A)\cup (<g>\cap I)\notin\mathcal{I}$. Suppose $(<g>\cap A)\in\mathcal{I}$. Now $<g>\cap I\subseteq I\Longrightarrow (<g>\cap A)\cup (<g>\cap I)\in\mathcal{I}$, contradiction and hence $<g>\cap A\notin\mathcal{I}\Longrightarrow g\in A^{\propto}$. Thus $(A\cup I)^{\propto}\subseteq A^{\propto}$. Therefore, $(A\cup I)^{\propto}= A^{\propto}$.

  For the second part:  $(A\setminus I)^{\propto}\subseteq A^{\propto}$ is obvious, as $(A\setminus I)\subseteq A$.

  Let $g\in A^{\propto}\Longrightarrow <g>\cap A\notin \mathcal{I}$. Now, $<g>\cap (A\setminus I)=<g>\cap A\setminus (<g>\cap I)$. Suppose $<g>\cap A\setminus (<g>\cap I)\in\mathcal{I}\Longrightarrow <g>\cap A=J\cup (<g>\cap I)\in\mathcal{I}$, contradiction. Thus, $<g>\cap (A\setminus I)\notin\mathcal{I}\Longrightarrow g\in(A\setminus I)^{\propto}$. Therefore, $A^{\propto}\subseteq (A\setminus I)^{\propto}$. Combining both, $(A\cup I)^{\propto}=A^{\propto}=(A\setminus I)^{\propto}$.
\end{proof}

\begin{Theorem}
Let $A$ be subset of a group $G$. Then $(A^{\propto})^{\propto} \subseteq A^{\propto}$.
\end{Theorem}

\begin{proof}
  Let $x \in  (A^{\propto})^{\propto}$. Then $<x> \cap A^{\propto} \notin \mathcal{I}$, and hence $<x> \cap A^{\propto} \neq \varnothing$. Suppose $y\in <x> \cap A^{\propto} $. Then $y\in <x>$ and $y\in A^{\propto}$. This implies that $<y> \cap A \notin \mathcal{I} \Longrightarrow <x>\cap A \notin \mathcal{I}$ (as, $<y> \subseteq <x>$). This implies that, $x\in A^{\propto}$. Hence,  $(A^{\propto})^{\propto} \subseteq A^{\propto}$.
\end{proof}

At the end of this section, following is an interesting lemma to construct a topology on a group.

\begin{Lemma}
  Let $\mathcal{I}$ be an ideal on the ring $2^G$, where $G$ is a group. Then the set operator, $\zeta: 2^G \rightarrow 2^G$ defined by
   $\zeta(A) = A\cup A^{\propto}$ satisfies the Kuratowski's closure axioms.
\end{Lemma}

\begin{Corollary}
 Let $\mathcal{I}$ be an ideal on the ring $2^G$, where $G$ is a group. Then $\{ H \in 2^G:\;  \zeta(G\setminus H) = (G\setminus H) \}$ forms a topology on $G$.
\end{Corollary}

\section{Anti-ideal on a ring}

In this section we shall introduce a new subset on a ring, and then we discuss the properties of this set through ideal.

\begin{Definition}
  Let $(R,+,.)$ be a ring. A subset $A\subseteq R$ is called an anti-ideal if
  \begin{enumerate}[(i)]
    \item for each $a, b\in A$, $a-b\notin A$
    \item for each $a, b\in A$ and $a\neq b$, $(a+b)+a.b \notin A$
  \end{enumerate}
  It is obvious that anti-ideal does not contain the identity element.
\end{Definition}

\begin{Example}
  For the ring $(\mathbb{Z},+,.)$, each singleton subset of $\mathbb{Z}$, $\{3, 5 \}$ and $\{3, 5, 7 \}$ are anti-ideals on $\mathbb{Z}$.
\end{Example}

\begin{Example}
Consider the ring $R=C[0,1]$ \cite{DAMIT2024,GILLMAN}, the set of all continuous functions on $[0, 1]$. Then $\{x,y\}=A$  is an anti-ideal on $R$, where
\begin{center}
  $x(t)=\begin{cases}
                                                              0, & \mbox{if }\; 0\leq t\leq\frac{1}{2} \\
                                                              t-\frac{1}{2}, & \mbox{if}\; \frac{1}{2}\leq t\leq 1.
                                                            \end{cases}$ \\  \& \\
                                                             $y(t)=\begin{cases}
                                                               \frac{1}{2}-t, & \mbox{if }\; 0\leq t\leq\frac{1}{2}\\
                                                            0 , & \mbox{if} \; \frac{1}{2} \leq t\leq 1.
                                                            \end{cases}$
\end{center}

\begin{figure}[h]
    \centering
    \begin{subfigure}{0.3\textwidth}
        \centering
        \includegraphics[width=\linewidth]{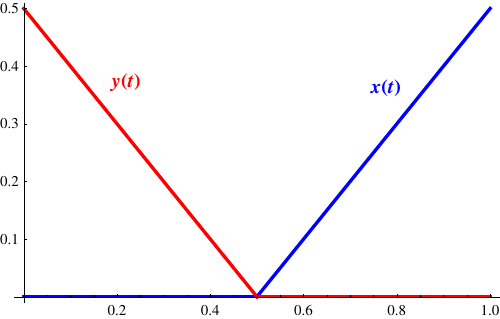}

    \end{subfigure}
    \hfill
    \begin{subfigure}{0.3\textwidth}
        \centering
        \includegraphics[width=\linewidth]{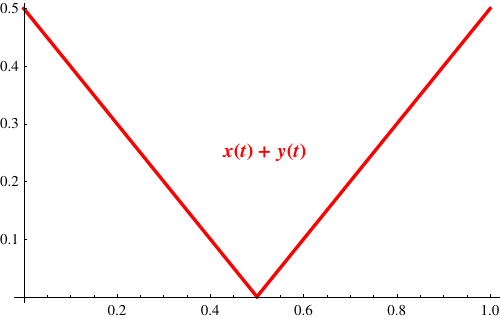}

    \end{subfigure}
    \hfill
    \begin{subfigure}{0.3\textwidth}
        \centering
        \includegraphics[width=\linewidth]{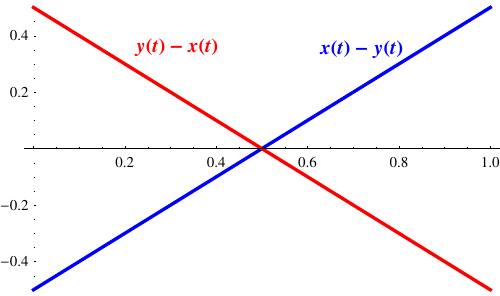}

    \end{subfigure}

    \caption{}
\end{figure}

\end{Example}

\newpage

The figure shows that addition and the compositions of addition and multiplication are not the form of $x(t)$ or $y(t)$.

Further, we will be mentioned that the anti-ideals on a ring are not limited, because for $a\in (0, 1)$, $\{x_a,\; y_a \}$ is an anti-ideal on $C[0,1]$, where

\begin{center}
  $x_a(t)=\begin{cases}
                                                              0, & \mbox{if }\; 0\leq t\leq a \\
                                                              t- a, & \mbox{if}\; a\leq t\leq 1.
                                                            \end{cases}$ \\  \& \\
                                                             $y_a(t)=\begin{cases}
                                                               a-t, & \mbox{if }\; 0\leq t\leq a\\
                                                            0 , & \mbox{if} \; a \leq t\leq 1 .
                                                            \end{cases}$
\end{center}

Note that the quotient with respect to an anti-ideal is not necessarily a ring again.

Anti-ideal of the power set ring $2^G$ has been characterized by the following lemma.

\begin{Lemma}
   ${}^{a}\mathcal{I}$ is an anti-ideal on the power set ring $2^{G}$ iff for $A\varsubsetneqq B\in {}^{a}\mathcal{I}, A\notin {}^{a}\mathcal{I}$ and for  distinct $A, B\in {}^{a}\mathcal{I}, A\cup B\notin {}^{a}\mathcal{I}$.
\end{Lemma}

\begin{proof}
 Second condition of the anti-ideal can be characterized by the relation, $[(A\setminus B)\cup (B\setminus A)]\Delta (A\cap B) = A\cup B \notin {}^{a}\mathcal{I}$.

 For the first condition, suppose ${}^{a}\mathcal{I}$ an anti-ideal.
 Let $A\varsubsetneqq B \in  {}^{a}\mathcal{I}$. If $A\in {}^{a}\mathcal{I}$, then $B = A\cup B \in {}^{a}\mathcal{I}$, a contradiction. Thus   $A \notin  {}^{a}\mathcal{I}$. Next suppose $A\varsubsetneqq B \in  {}^{a}\mathcal{I}$ implies $A\notin {}^{a}\mathcal{I}$. Now, if $A\Delta B \in  {}^{a}\mathcal{I}$, then $B\setminus A \in {}^{a}\mathcal{I}$. Since $B\setminus A$ is a subset of $B$, then the last relation is a contradiction. Therefore, $A\Delta B \notin  {}^{a}\mathcal{I}$
\end{proof}

\begin{Proposition}
  Let $f: R\rightarrow R'$ be a ring homomorphism. Then $f({}^{a}\mathcal{I})$ may not be an anti-ideal on $R'$, where ${}^{a}\mathcal{I}$ is an anti-ideal on $R$ and $f({}^{a}\mathcal{I}) = \{f({}^{a}I):\; {}^{a}I \in  {}^{a}\mathcal{I} \}$.
\end{Proposition}

\begin{Example}
   Let $R=C[0,1]$ and let $f:R\rightarrow R$ defined by $f(x)=0$ be a homomorphism. Then for any anti ideal ${}^{a}\mathcal{I}$, $f({}^{a}\mathcal{I})=\{\{0\}\}$ is not an anti ideal.
\end{Example}

\begin{Theorem}
  Let $f: R\rightarrow R'$ be a ring isomorphism. Then, for an anti-ideal ${}^{a}\mathcal{I}$ on $R$, $f({}^{a}\mathcal{I})$ is an anti-ideal on $R'$, where ${}^{a}\mathcal{I}$ is an anti-ideal on $R$.
\end{Theorem}

\begin{proof}
  Let $A',B'\in f({}^{a}\mathcal{I})$. Then, $f(A)=A'$ and $f(B)=B'$, where $A,\; B\in {}^{a}\mathcal{I}$. Now $A' \Delta B' = f(A) \Delta f(B) = (f(A)\setminus f(B))\cup (f(B)\setminus f(A)) = f((A\setminus B)\cup (B\setminus A))$, as $f$ is bijective. Since $A\Delta B \notin {}^{a}\mathcal{I}$, then $f(A\Delta B) \notin f({}^{a}\mathcal{I})$.  Next, $A\cup B = (A\Delta B)\Delta (A\cap B) \notin {}^{a}\mathcal{I}$, and hence, $f(A\cup B) = f(A)\cup f(B) = (f(A)\Delta f(B))\Delta (f(A)\cap f(B)) \notin f({}^{a}\mathcal{I})$.

\end{proof}

Local function due to the anti-ideal:
\begin{Definition}
  Let $G$ be a group and $A\subseteq G$. Define $A^{a\propto}=\{g\in G: H\cap A\in {}^{a}\mathcal{I}, \;\forall\; H\in \mathcal{G}(g)\}$.
\end{Definition}

This operator $( )^{a\propto}: 2^G \rightarrow 2^G$ is called anti-ideal operator.

\begin{Example}
  Let $G=\mathbb{Z}_{4}$ and ${}^{a}\mathcal{I}=\{\{1\}\}$. Then, for $A=\{1,2\}$, $A^{a\propto}=\varnothing$.
\end{Example}

\begin{Example}
  Let $G=\mathbb{Z}$ and ${}^{a}\mathcal{I}=\{\{0\}\}$. For $A=\{0,1\}$, $A^{a\propto}=\mathbb{Z}\setminus \{-1,0,1\}$.
\end{Example}

\begin{Proposition}
  Let $G$ be a group and $A\subseteq G$ be any set. Then for two anti-ideals ${}^{a}\mathcal{I}$, ${}^{a}\mathcal{J}$ with ${}^{a}\mathcal{I}\subseteq {}^{a}\mathcal{J}$, $A^{a\propto}({}^{a}\mathcal{I})\subseteq A^{a\propto}({}^{a}\mathcal{J})$.
\end{Proposition}

\begin{Proposition}
  The anti-star operator do not follow the monotonicity, that is, if $A\subseteq B$, then $A^{a\propto}\subseteq B^{a\propto}$ may or may not be hold.
\end{Proposition}

\begin{Example}
  Let $G=\mathbb{Z}_{4}$ and ${}^{a}\mathcal{I}=\{\{0\},\{2\}\}$. Take $A=\{0\}$ and $B=\{0,2\}$, then $A^{a\propto}=\mathbb{Z}_{4}$ and $B^{a\propto}=\varnothing$.
\end{Example}

\begin{Proposition}
  The anti-star operator does not preserve under union operation, that is, $(A\cup B)^{a\propto}\neq A^{a\propto}\cup B^{a\propto}$.
\end{Proposition}

\begin{Example}
 Let $G=\mathbb{Z}$ and ${}^{a}\mathcal{I}=\{\{0,1\}\}$. Then for $A=\{0\}$ and $B=\{1\}$, $A^{a\propto}=\varnothing$ and $B^{a\propto}=\varnothing$, but $(A\cup B)^{a\propto}=\{0,1\}^{a\propto}=\{-1,0,1\}$. Thus, $(A\cup B)^{a\propto}\nsubseteq A^{a\propto}\cup B^{a\propto}$.

  Again, if we take the anti-ideal ${}^{a}\mathcal{I}=\{\{0\}\}$ and $A=\{0\}$, $B=\{2\}$, then $A^{a\propto}=\mathbb{Z}$ and $B^{a\propto}=\varnothing$, but $(A\cup B)^{a\propto}=\{0,2\}^{a\propto}=\mathbb{Z}\setminus\{-2,-1,0,1,2\}$. Thus, $A^{a\propto}\cup B^{a\propto}\nsubseteq (A\cup B)^{a\propto}$.
\end{Example}

\begin{Proposition}
  The anti-star operator does not preserve under intersection operation,  that is, $A^{a\propto}\cap B^{a\propto}\neq (A\cap B)^{a\propto}$.
\end{Proposition}

\begin{Example}
  Let $G=\mathbb{Z}$ and ${}^{a}\mathcal{I}=\{\{0\}\}$. Then for $A=\{0,1\}$ and $B=\{0,2\}$, $A^{a\propto}=\mathbb{Z} \setminus \{-1, 0, 1 \}$, and $B^{a\propto}= \mathbb{Z} \setminus \{-2, -1, 0, 1, 2 \}$, then $A^{a\propto}\cap B^{a\propto}=\mathbb{Z} \setminus \{-2, -1, 0, 1, 2 \}$, but $(A\cap B)^{a\propto}=\{0\}^{a\propto}=\mathbb{Z}$. Thus, $(A\cap B)^{a\propto}\nsubseteq A^{a\propto}\cap B^{a\propto}$.
\end{Example}

\begin{Proposition}
  Let $G$ be any group and ${}^{a}\mathcal{I}$ be an anti-ideal on $2^{G}$. Then, for any subset $A$ of $G$ and ${}^{a}I\in {}^{a}\mathcal{I}$, $A^{a\propto}\neq(A\setminus I)^{a\propto}\neq(A\cup I)^{a\propto}$.
\end{Proposition}

\begin{Example}
  Let $G=\mathbb{Z}$ and ${}^{a}\mathcal{I}=\{\{0\},\{2\}\}$. Then, for $A=\{0,2\}$ and ${}^{a}I=\{0\}$, $A^{a\propto}=\mathbb{Z} \setminus \{-2, -1, 0, 1, 2 \}$, but $(A\setminus {}^{a}I)^{a\propto}=\{2\}^{a\propto}=\{-2, -1, 0, 1, 2 \}$. Also, if we take $B=\{2\}$, then $B^{a\propto}=\{-2, -1, 0, 1, 2 \}$, but $(B\cup {}^{a}I)^{a\propto}=\{0,2\}^{a\propto}=\mathbb{Z} \setminus \{-2, -1, 0, 1, 2 \}$.
\end{Example}

\begin{Note}
  Let $f:G\rightarrow G$ be a group homomorphism and ${}^{a}\mathcal{I}$ is an anti-ideal on $2^{G}$. Then for any subset $A$ of $G$, $f(A^{a\propto}({}^{a}\mathcal{I}))\neq(f(A))^{a\propto}({}^{a}\mathcal{I})$.
\end{Note}

\begin{Example}
  Let $f:\mathbb{Z}\rightarrow \mathbb{Z}$ defined by $f(n)=2n$ be a group homomorphism. Take the anti-ideal ${}^{a}\mathcal{I}=\{\{0\}\}$ and $A=\{0,1\}$. Then $A^{a\propto}({}^{a}\mathcal{I})=\mathbb{Z}\setminus \{-1, 0, 1\}$, $f(A^{a\propto}({}^{a}\mathcal{I}))=2\mathbb{Z}\setminus \{-2, 0,2\}$. Now, $f(A)=\{0,2\}$, $(f(A))^{a\propto}({}^{a}\mathcal{I})=\{0,2\}^{a\propto}({}^{a}\mathcal{I})=\mathbb{Z}\setminus\{-2, -1, 0, 1, 2\}$.
\end{Example}

{\bf Conclusion}: The paper has been constructed two ideas, one is topology on a group and another is, anti-ideal on a ring. This paper  introduced a new branch in the combine field of group, ring and topology.

\end{document}